\documentclass[preprint,12pt]{elsarticle}
\usepackage{amssymb}
\usepackage{amsthm}
\usepackage{amsmath}
\usepackage{amsfonts}

\newtheorem{theorem}{Theorem}

\newtheorem{corollary}{Corollary}
\newtheorem{example}{Example}

\usepackage{ragged2e}

\journal{arXiv}

\begin{document}

\begin{frontmatter}



\title{An Explicit Identity to Solve Sums of\\ Powers of Complex Functions}


\author{Dagnachew Jenber Negash}

\address{Addis Ababa Science and Technology University\\Addis Ababa, Ethiopia\\Email: djdm$\_$101979@yahoo.com}
\begin{abstract}
A recurrence relations for sums of powers of complex functions can be written as a system of linear equation $AX=B$. Using properties of determinant and Cramer's rule for solving systems of linear equation, this paper presents an absolutely explicit identity for solving sums of powers of complex functions with out one sum depends on the others.
\end{abstract}


\end{frontmatter}


\section{Introduction}
\label{S:1}
\justify
Students often encounter formulas for sums of powers of the first $n$ positive integers as examples of statements that can be proved using the Principle of Mathematical Induction and, perhaps less often nowadays, in Riemann sums during an introduction to definite integration. In either situation, they usually see only the first three such sum formulas,
\begin{equation*}
1+2+3+\cdots+n=\frac{n(n+1)}{2}
\end{equation*}

\begin{equation*}
1^2+2^2+3^2+\cdots+n^2=\frac{n(n+1)(2n+1)}{6}
\end{equation*}and
\begin{equation*}
1^3+2^3+3^3+\cdots+n^3=\frac{n^2(n+1)^2}{4}
\end{equation*}
for any positive integer $n$.\\[3mm]
Formulas for sums of integer powers were first given in generalizable form in the West by Thomas Harriot (c. 1560-1621) of England. At about the same time, Johann Faulhaber (1580-1635) of Germany gave formulas for these sums up to the $17^{th}$ power, far higher than anyone before him, but he did not make clear how to generalize them. Pierre de Fermat (1601-1665) often is credited with the discovery of formulas for sums of integer powers, but his fellow French mathematician Blaise Pascal (1623-1662) gave the formulas much more explicitly. The Swiss mathematician Jakob Bernoulli (1654-1705) is perhaps best and most deservedly known for presenting formulas for sums of integer powers to the European mathematical community. His was the most useful and generalizable formulation to date because he gave by far the most explicit and succinct instructions for finding the coefficients of the formulas.\\
Generally, the present paper provided an identity to solve one sums of powers of complex functions with out depending on the other sums, refer to Section \ref{S:3}.


\section{Preliminary}
\label{S:2}
In the following Theorem $1$ and $3$, formulas $(1)$ and $(2)$ generalize the well known formulas for $a=d=1$\cite{8}. We note that equation $(2)$ was found by Wiener\cite{10}; see also\cite{2}. Bachmann\cite{1} found a recurrence for $L_{k,n}(a,d)$ involving only $L_{j,n}$ for $j=1,2,3,\cdots,k-1.$ This paper presents a formula for $L_{k,n}(a,d)$ and $T_{k,n}(a,d)$ from equation $(1)$ and $(2)$ without the involvement of $L_{j,n}$ for $j=1,2,3,\cdots,k-1$ and $T_{j,n}$ for $j=1,2,3,\cdots,k-1$ respectively. 
\begin{theorem}
For all $n,k \in \mathbb{N}$, and $a,d\in{\mathbb{C}}$, where $d\ne0$
\begin{equation}
\sum_{j=0}^{k}\binom{k+1}jd^{k+1-j}L_{j,n}(a,d)=(a+nd)^{k+1}-(a)^{k+1}
\end{equation}Where
\begin{equation*}
\binom {k+1} j=\frac{(k+1)!}{j!(k+1-r)!}
\end{equation*}and
\begin{equation*}
L_{j,n}(a,d)=a^j+(a+d)^j+(a+2d)^j+\cdots+(a+(n-1)d)^j
\end{equation*}

\end{theorem}

\begin{theorem}
The above theorem can be rewritten in the following system of equation form for $n=1,2,3,\cdots$
\begin{equation*}
AX=B
\end{equation*}Where 
\begin{equation*}
A=\begin{pmatrix}
\binom 1 0 d & 0 & 0 & 0 & 0 &  0 & 0 & 0 &0 & 0\\[2mm]
\binom 2 0 d^2& \binom 2 1 d& 0 & 0 & 0 &  0 & 0 & 0 & 0 & 0\\[2mm]
\binom 3 0 d^3& \binom 3 1 d^2& \binom 3 2 d & 0 & 0 &  0 & 0 & 0 & 0 & 0\\[2mm]
\binom 4 0 d^4& \binom 4 1 d^3 & \binom 4 2 d^2 & \binom 4 3 d & 0 & 0 & 0 & 0 & 0 & 0
\\[2mm]
\binom 5 0 d^5& \binom 5 1 d^4 &\binom 5 2 d^3 & \binom 5 3 d^2 & \binom 5 4 d & 0 & 0 & 0 & 0 & 0\\[2mm]
. & . &. &.&.& .& 0 & 0 & 0 &0\\[2mm]
. & . &. &.&.& . & . & 0 & 0 &0\\[2mm]
. & . &. &.&.& . & . & . & 0 &0\\[2mm]
\binom {k} 0d^k & \binom {k} 1 d^{k-1}&\binom {k} 2 d^{k-2} & \binom {k} 3 d^{k-3} & . & . & . & .  & \binom {k} {k-1}d & 0\\[2mm]
\binom {k+1} 0 d^{k+1} & \binom {k+1} 1 d^{k}&\binom {k+1} 2 d^{k-1}& \binom {k+1} 3 d^{k-2} & . & . & . & . &\binom {k+1} {k-1}d^2 & \binom {k+1} {k}d\\[2mm]
\end{pmatrix}
\end{equation*}
\begin{equation*}
X=
\begin{pmatrix}
L_{0,n}(a,d)\\[2mm]
L_{1,n}(a,d)\\[2mm]
L_{2,n}(a,d)\\[2mm]
L_{3,n}(a,d)\\[2mm]
. \\[2mm]
.\\[2mm]
.\\[2mm]
L_{k-1,n}(a,d)\\[2mm]
L_{k,n}(a,d)
\end{pmatrix}
 \text{ , }
B=\begin{pmatrix}
(a+nd)-a\\[2mm]
(a+nd)^{2}-a^{2}\\[2mm]
(a+nd)^{3}-a^{3}\\[2mm]
(a+nd)^{4}-a^{4}\\[2mm]
.\\[2mm]
.\\[2mm]
.\\[2mm]
(a+nd)^{k}-a^{k}\\[2mm]
(a+nd)^{k+1}-a^{k+1}
\end{pmatrix}
\end{equation*}
Thus to solve for $X$ we can use cramer's rule of solving linear systems of equation, since $det(A)\neq 0$
\end{theorem}
\begin{theorem}
For all $n,k \in \mathbb{N}$, and $a,d\in{\mathbb{C}}$, where $d\ne0$
\begin{equation}
\sum_{j=0}^{k}(-1)^j\binom{k+1}jd^{k+1-j}T_{j,n}(a,d)=(-1)^k\bigg[(a+nd-d)^{k+1}-(a-d)^{k+1}\bigg]
\end{equation}Where
\begin{equation*}
\binom {k+1} j=\frac{(k+1)!}{j!(k+1-r)!}
\end{equation*}and
\begin{equation*}
T_{j,n}(a,d)=a^j-(a+d)^j+(a+2d)^j-\cdots+(-1)^{(n-1)}(a+(n-1)d)^j
\end{equation*}

\end{theorem}

\section{Main Result}
\label{S:3}
\begin{theorem}
For all $k=2,3,4,\cdots$ and $n=k+1,k+2,k+3,\cdots$
\begin{equation}
S^{k-2}_n=-\binom n {k-1}\frac{1}{k}d^{n-k}S^{k-3}_k+S^{k-3}_{n}
\end{equation}Where
\begin{equation*}
S^0_{n}=S_n=\bigg(\frac{n}{2}-1\bigg)kd^{n}-\frac{n}{2}d^{n-2}((a+kd)^2-a^2)+(a+kd)^{n}-a^{n}
\end{equation*}
\end{theorem}

\begin{proof}Let $J^r=(a+nd)^r-a^r$ for $r=1,2,3,\cdots,k+1$ and
\begin{equation*}
M_1=\begin{pmatrix}
\binom 1 0 d & 0 & 0 & 0 & 0 &  0 & 0 & 0 &0 & J\\[2mm]
\binom 2 0 d^2& \binom 2 1 d& 0 & 0 & 0 &  0 & 0 & 0 & 0 & J^2\\[2mm]
\binom 3 0 d^3& \binom 3 1 d^2& \binom 3 2 d & 0 & 0 &  0 & 0 & 0 & 0 & J^3\\[2mm]
\binom 4 0 d^4& \binom 4 1 d^3 & \binom 4 2 d^2 & \binom 4 3 d & 0 & 0 & 0 & 0 & 0 & J^4
\\[2mm]
\binom 5 0 d^5& \binom 5 1 d^4 &\binom 5 2 d^3 & \binom 5 3 d^2 & \binom 5 4 d & 0 & 0 & 0 & 0 & J^5\\[2mm]
. & . &. &.&.& .& 0 & 0 & 0 &.\\[2mm]
. & . &. &.&.& . & . & 0 & 0 &.\\[2mm]
. & . &. &.&.& . & . & . & 0 &.\\[2mm]
\binom {k} 0d^k & \binom {k} 1 d^{k-1}&\binom {k} 2 d^{k-2} & \binom {k} 3 d^{k-3} & . & . & . & .  & \binom {k} {k-1}d & J^k\\[2mm]
\binom {k+1} 0 d^{k+1} & \binom {k+1} 1 d^{k}&\binom {k+1} 2 d^{k-1}& \binom {k+1} 3 d^{k-2} & . & . & . & . &\binom {k+1} {k-1}d^2 & J^{k+1}\\[2mm]
\end{pmatrix}
\end{equation*}
Let us apply the following elementary row operations to determine the determinant of matrix $M_1$. Apply
\begin{equation*}
 -d^{n-1}R_1+R_n\longrightarrow R_n \text{ for }n=2,3,\cdots k+1.
\end{equation*}
Then we get the matrix
\begin{equation*}
M_2=\begin{pmatrix}
\binom 1 0 d & 0 & 0 & 0 & 0 &  0 & 0 & 0 &0 & J\\[2mm]
0& \binom 2 1 d& 0 & 0 & 0 &  0 & 0 & 0 & 0 & -d^{1}J+J^2\\[2mm]
0& \binom 3 1 d^2& \binom 3 2 d & 0 & 0 &  0 & 0 & 0 & 0 & -d^{2}J+J^3\\[2mm]
0& \binom 4 1 d^3 & \binom 4 2 d^2 & \binom 4 3 d & 0 & 0 & 0 & 0 & 0 & -d^{3}J+J^4
\\[2mm]
0& \binom 5 1 d^4 &\binom 5 2 d^3 & \binom 5 3 d^2 & \binom 5 4 d & 0 & 0 & 0 & 0 & -d^{4}J+J^5\\[2mm]
. & . &. &.&.& .& 0 & 0 & 0 &.\\[2mm]
. & . &. &.&.& . & . & 0 & 0 &.\\[2mm]
. & . &. &.&.& . & . & . & 0 &.\\[2mm]
0& \binom {k} 1 d^{k-1}&\binom {k} 2 d^{k-2} & \binom {k} 3 d^{k-3} & . & . & . & .  & \binom {k} {k-1}d & -d^{k-1}J+J^k\\[2mm]
0 & \binom {k+1} 1 d^{k}&\binom {k+1} 2 d^{k-1}& \binom {k+1} 3 d^{k-2} & . & . & . & . &\binom {k+1} {k-1}d^2 & -d^{k}J+J^{k+1}\\[2mm]
\end{pmatrix}
\end{equation*}
Apply the elementary row opertion on matrix $M_2$:
\begin{equation*}
 -\binom n 1\frac{1}{2}d^{n-2} R_2+R_n\longrightarrow R_n \text{ for }n=3,4,\cdots k+1. 
\end{equation*}
To get the matrix:
\begin{equation*}
\begin{pmatrix}
\binom 1 0 d & 0 & 0 & 0 & 0 &  0 & 0 & 0 &0 & J\\[2mm]
0& \binom 2 1 d& 0 & 0 & 0 &  0 & 0 & 0 & 0 & -d^{1}J+J^2\\[2mm]
0& 0& \binom 3 2 d & 0 & 0 &  0 & 0 & 0 & 0 &-\frac{3}{2}d^{1}(-d^{1}J+J^2) -d^{2}J+J^3\\[2mm]
0& 0& \binom 4 2 d^2 & \binom 4 3 d & 0 & 0 & 0 & 0 & 0 & -\frac{4}{2}d^{2}(-d^{1}J+J^2)-d^{3}J+J^4
\\[2mm]
0& 0 &\binom 5 2 d^3 & \binom 5 3 d^2 & \binom 5 4 d & 0 & 0 & 0 & 0 &-\frac{5}{2}d^{3}(-d^{1}J+J^2)-d^{4}J+J^5\\[2mm]
. & . &. &.&.& .& 0 & 0 & 0 &.\\[2mm]
. & . &. &.&.& . & . & 0 & 0 &.\\[2mm]
. & . &. &.&.& . & . & . & 0 &.\\[2mm]
0& 0&\binom {k} 2 d^{k-2} & \binom {k} 3 d^{k-3} & . & . & . & .  & \binom {k} {k-1}d & -\frac{k}{2}d^{k-2}(-d^{1}J+J^2)-d^{(k-1)}J+J^k\\[2mm]
0 & 0&\binom {k+1} 2 d^{k-1}& \binom {k+1} 3 d^{k-2} & . & . & . & . &\binom {k+1} {k-1}d^2 & -\frac{k+1}{2}d^{k-1}(-d^{1}J+J^2)-d^{k}J+J^{k+1}\\[2mm]
\end{pmatrix}
\end{equation*}

Let
\begin{equation*}
S_1=J
\end{equation*}
\begin{equation*}
S_2=-d^{1}J+J^2
\end{equation*}
\begin{equation*}
S_3= \frac{3}{2}d^{2}J-\frac{3}{2}d^{1}J^2-d^{2}J+J^3
\end{equation*}
\begin{equation*}
S_4=\frac{4}{2}d^{3}J-\frac{4}{2}d^{2}J^2-d^{3}J+J^4
\end{equation*}
\begin{equation*}
S_5=\frac{5}{2}d^{4}J-\frac{5}{2}d^{3}J^2 -d^{4}J+J^5
\end{equation*}

\begin{equation*}
\cdots
\end{equation*}
\begin{equation*}
\cdots
\end{equation*}
\begin{equation*}
\cdots
\end{equation*}
\begin{equation*}
S_{k}=\frac{k}{2}d^{k-1}J-\frac{k}{2}d^{k-2}J^2-d^{k-1}J+J^k
\end{equation*}
\begin{equation*}
S_{(k+1)}=\frac{k+1}{2}d^{k}J-\frac{k+1}{2}d^{k-1}J^2-d^{k}J+J^{k+1}
\end{equation*}
Therefore matrix $M_4$ becomes
\begin{equation*}M_4=
\begin{pmatrix}
\binom 1 0 d & 0 & 0 & 0 & 0 &  0 & 0 & 0 &0 & S_1\\[2mm]
0& \binom 2 1 d& 0 & 0 & 0 &  0 & 0 & 0 & 0 & S_2\\[2mm]
0& 0& \binom 3 2 d & 0 & 0 &  0 & 0 & 0 & 0 & S_3\\[2mm]
0& 0& \binom 4 2 d^2 & \binom 4 3 d & 0 & 0 & 0 & 0 & 0 & S_4
\\[2mm]
0& 0 &\binom 5 2 d^3 & \binom 5 3 d^2 & \binom 5 4 d & 0 & 0 & 0 & 0 &S_5\\[2mm]
. & . &. &.&.& .& 0 & 0 & 0 &.\\[2mm]
. & . &. &.&.& . & . & 0 & 0 &.\\[2mm]
. & . &. &.&.& . & . & . & 0 &.\\[2mm]
0& 0&\binom {k} 2 d^{k-2} & \binom {k} 3 d^{k-3} & . & . & . & .  & \binom {k} {k-1}d & S_k\\[2mm]
0 & 0&\binom {k+1} 2 d^{k-1}& \binom {k+1} 3 d^{k-2} & . & . & . & . &\binom {k+1} {k-1}d^2 & S_{(k+1)}\\[2mm]
\end{pmatrix}
\end{equation*}
Apply the elementary row opertion on matrix $M_4$:
\begin{equation*}
 -\binom n 2\frac{1}{3}d^{n-3} R_3+R_n\longrightarrow R_n \text{ for }n=4,5,\cdots k+1. 
\end{equation*}
Then we get the matrix
\begin{equation*}M_5=
\begin{pmatrix}
\binom 1 0 d & 0 & 0 & 0 & 0 &  0 & 0 & 0 &0 & S_1\\[2mm]
0& \binom 2 1 d& 0 & 0 & 0 &  0 & 0 & 0 & 0 & S_2\\[2mm]
0& 0& \binom 3 2 d & 0 & 0 &  0 & 0 & 0 & 0 & S_3\\[2mm]
0& 0& 0 & \binom 4 3 d & 0 & 0 & 0 & 0 & 0 & -\binom 4 2\frac{1}{3}d^{1} S_3+S_4
\\[2mm]
0& 0 &0 & \binom 5 3 d^2 & \binom 5 4 d & 0 & 0 & 0 & 0 &-\binom 5 2\frac{1}{3}d^{2} S_3+S_5\\[2mm]
. & . &. &.&.& .& 0 & 0 & 0 &.\\[2mm]
. & . &. &.&.& . & . & 0 & 0 &.\\[2mm]
. & . &. &.&.& . & . & . & 0 &.\\[2mm]
0& 0&0 & \binom {k} 3 d^{k-3} & . & . & . & .  & \binom {k} {k-1}d & -\binom k 2\frac{1}{3}d^{k-3} S_3+S_k\\[2mm]
0 & 0&0& \binom {k+1} 3 d^{k-2} & . & . & . & . &\binom {k+1} {k-1}d^2 & -\binom {k+1} 2\frac{1}{3}d^{k-2} S_3+S_{(k+1)}\\[2mm]
\end{pmatrix}
\end{equation*}
Let
\begin{equation*}
S^1_4=-\binom 4 2\frac{1}{3}d^{1} S_3+S_4
\end{equation*}
\begin{equation*}
S^1_5=-\binom 5 2\frac{1}{3}d^{2} S_3+S_5
\end{equation*}

\begin{equation*}
\cdots
\end{equation*}
\begin{equation*}
\cdots
\end{equation*}
\begin{equation*}
\cdots
\end{equation*}
\begin{equation*}
S^1_{k}=-\binom k 2\frac{1}{3}d^{k-3} S_3+S_k
\end{equation*}
\begin{equation*}
S^1_{(k+1)}=-\binom {k+1} 2\frac{1}{3}d^{k-2} S_3+S_{(k+1)}
\end{equation*}
Therefore 
\begin{equation*}M_5=
\begin{pmatrix}
\binom 1 0 d & 0 & 0 & 0 & 0 &  0 & 0 & 0 &0 & S_1\\[2mm]
0& \binom 2 1 d& 0 & 0 & 0 &  0 & 0 & 0 & 0 & S_2\\[2mm]
0& 0& \binom 3 2 d & 0 & 0 &  0 & 0 & 0 & 0 & S_3\\[2mm]
0& 0& 0 & \binom 4 3 d & 0 & 0 & 0 & 0 & 0 & S^1_4
\\[2mm]
0& 0 &0 & \binom 5 3 d^2 & \binom 5 4 d & 0 & 0 & 0 & 0 &S^1_5\\[2mm]
. & . &. &.&.& .& 0 & 0 & 0 &.\\[2mm]
. & . &. &.&.& . & . & 0 & 0 &.\\[2mm]
. & . &. &.&.& . & . & . & 0 &.\\[2mm]
0& 0&0 & \binom {k} 3 d^{k-3} & . & . & . & .  & \binom {k} {k-1}d & S^1_k\\[2mm]
0 & 0&0& \binom {k+1} 3 d^{k-2} & . & . & . & . &\binom {k+1} {k-1}d^2 & S^1_{(k+1)}\\[2mm]
\end{pmatrix}
\end{equation*}
Apply the elementary row opertion on matrix $M_5$:
\begin{equation*}
 -\binom n 3\frac{1}{4}d^{n-4}R_4+R_n\longrightarrow R_n \text{ for }n=5,6,\cdots k+1. 
\end{equation*}
Then we get the matrix
\begin{equation*}M_6=
\begin{pmatrix}
\binom 1 0 d & 0 & 0 & 0 & 0 &  0 & 0 & 0 &0 & S_1\\[2mm]
0& \binom 2 1 d& 0 & 0 & 0 &  0 & 0 & 0 & 0 & S_2\\[2mm]
0& 0& \binom 3 2 d & 0 & 0 &  0 & 0 & 0 & 0 & S_3\\[2mm]
0& 0& 0 & \binom 4 3 d & 0 & 0 & 0 & 0 & 0 & S^1_4
\\[2mm]
0& 0 &0 & 0 & \binom 5 4 d & 0 & 0 & 0 & 0 &-\binom 5 3\frac{1}{4}d^{1} S^1_4+S^1_5\\[2mm]
. & . &. &.&.& .& 0 & 0 & 0 &.\\[2mm]
. & . &. &.&.& . & . & 0 & 0 &.\\[2mm]
. & . &. &.&.& . & . & . & 0 &.\\[2mm]
0& 0&0 & 0 & . & . & . & .  & \binom {k} {k-1}d &-\binom k 3\frac{1}{4}d^{k-4} S^1_4 +S^1_k\\[2mm]
0 & 0&0& 0 & . & . & . & . &\binom {k+1} {k-1}d^2 &-\binom {k+1} 3\frac{1}{4}d^{k-3} S^1_4+ S^1_{(k+1)}\\[2mm]
\end{pmatrix}
\end{equation*}

Let
\begin{equation*}
S^2_5=-\binom 5 3\frac{1}{4}d^{1} S^1_4+S^1_5
\end{equation*}
\begin{equation*}
S^2_6=-\binom 6 3\frac{1}{4}d^{2} S^1_4+S^1_6
\end{equation*}
\begin{equation*}
S^2_7=-\binom 7 3\frac{1}{4}d^{3} S^1_4+S^1_7
\end{equation*}
\begin{equation*}
S^2_8=-\binom 8 3\frac{1}{4}d^{4} S^1_4+S^1_8
\end{equation*}
\begin{equation*}
\cdots
\end{equation*}
\begin{equation*}
\cdots
\end{equation*}
\begin{equation*}
\cdots
\end{equation*}
\begin{equation*}
S^2_{k}=-\binom {k} 3\frac{1}{4}d^{k-4} S^1_4 +S^1_k
\end{equation*}
\begin{equation*}
S^2_{(k+1)}=-\binom {k+1} 3\frac{1}{4}d^{k-3} S^1_4+ S^1_{(k+1)}
\end{equation*}
Therefore 
\begin{equation*}M_6=
\begin{pmatrix}
\binom 1 0 d & 0 & 0 & 0 & 0 &  0 & 0 & 0 &0 & S_1\\[2mm]
0& \binom 2 1 d& 0 & 0 & 0 &  0 & 0 & 0 & 0 & S_2\\[2mm]
0& 0& \binom 3 2 d & 0 & 0 &  0 & 0 & 0 & 0 & S_3\\[2mm]
0& 0& 0 & \binom 4 3 d & 0 & 0 & 0 & 0 & 0 & S^1_4
\\[2mm]
0& 0 &0 & 0 & \binom 5 4 d & 0 & 0 & 0 & 0 &S^2_5\\[2mm]
. & . &. &.&.& .& 0 & 0 & 0 &.\\[2mm]
. & . &. &.&.& . & . & 0 & 0 &.\\[2mm]
. & . &. &.&.& . & . & . & 0 &.\\[2mm]
0& 0&0 & 0 & . & . & . & .  & \binom {k} {k-1}d &S^2_k\\[2mm]
0 & 0&0& 0 & . & . & . & . &\binom {k+1} {k-1}d^2 & S^2_{(k+1)}\\[2mm]
\end{pmatrix}
\end{equation*}
By applying Elementary row operation repeatedly, then we get the matrix
\begin{equation*}M_{k+2}=
\begin{pmatrix}
\binom 1 0 d & 0 & 0 & 0 & 0 &  0 & 0 & 0 &0 & S_1\\[2mm]
 0 & \binom 2 1 d& 0 & 0 & 0 &  0 & 0 & 0 & 0 & S_2\\[2mm]
 0 & 0& \binom 3 2 d& 0 & 0 &  0 & 0 & 0 & 0 & S_3\\[2mm]
0 & 0 & 0 & \binom 4 3d & 0 & 0 & 0 & 0 & 0 &S^1_4
\\[2mm]
0 & 0 & 0 & 0 & \binom 5 4 d& 0 & 0 & 0 & 0 &S^2_5\\[2mm]
. & . &. &.&.& .& 0 & 0 & 0 & S^3_6\\[2mm]
. & . &. &.&.& . & . & 0 & 0 & S^4_7\\[2mm]
. & . &. &.&.& . & . & . & 0 & S^5_8\\[2mm]
. & . &. &.&.& . & . & . & . & .\\[2mm]
. & . &. &.&.& . & . & . & . & .\\[2mm]
. & . &. &.&.& . & . & . & . & .\\[2mm]
 0 & 0 & 0 & 0 & . & . & . & .  & \binom {k} {k-1}d &S^{(k-3)}_{k}\\[2mm]
0 & 0 & 0 & 0 & . & . & . & . & 0 & S^{(k-2)}_{(k+1)}\\[2mm]
\end{pmatrix}
\end{equation*}
Therefore we get the identity for $k=3,4,5,\cdots$
\begin{equation*}
S^{k-2}_n=-\binom n {k-1}\frac{1}{k}d^{n-k}S^{k-3}_k+S^{k-3}_{n}\text{ for }n=k+1,k+2,k+3,\cdots
\end{equation*}
Hence from properties of determinant
\begin{equation*}
|M_1|=|M_{(k+3)}|=S^{(k-2)}_{(k+1)}d^k\prod_{i=1}^{k}(k+1-i)=k!d^{k}S^{(k-2)}_{(k+1)}
\end{equation*}
\end{proof}

\begin{theorem} For $n=4,5,6,7,\cdots$
\begin{equation}
S^{(n-3)}_n=\sum_{i=0}^{m}\binom m i \bigg(\frac{d}{2}\bigg)^i\frac{n!}{(n-i)!}(-1)^iS_{n-i}^{n-3-m}
\end{equation}
\end{theorem}
\begin{proof}
\begin{equation*}
S^{(n-3)}_n=-\frac{n}{2}dS_{n-1}^{n-4}+S_n^{n-4}=-\frac{n}{2}d\bigg[-\frac{n-1}{2}dS_{n-2}^{n-5}+S_{n-1}^{n-5}\bigg]-\frac{n}{2}dS_{n-1}^{n-5}+S_n^{n-5}
\end{equation*}
\begin{equation*}
=\frac{n(n-1)}{2^2}d^2S_{n-2}^{n-5}-\frac{n}{2}dS_{n-1}^{n-5}-\frac{n}{2}dS_{n-1}^{n-5}+S_n^{n-5}
\end{equation*}

\begin{equation*}
=\frac{n(n-1)}{2^2}d^2\bigg[-\frac{n-2}{2}dS_{n-3}^{n-6}+S_{n-2}^{n-6}\bigg]-2\frac{n}{2}d\bigg[ -\frac{n-1}{2}dS_{n-2}^{n-6}+S_{n-1}^{n-6} \bigg]-\frac{n}{2}dS_{n-1}^{n-6}+S_{n}^{n-6}
\end{equation*}

\begin{equation*}
=-\frac{n(n-1)(n-2)}{2^3}d^3S_{n-3}^{n-6}+\frac{n(n-1)}{2^2}d^2 S_{n-2}^{n-6}+2\bigg[\frac{n(n-1)}{2^2}d^2S_{n-2}^{n-6}-\frac{n}{2}d S_{n-1}^{n-6} \bigg]-\frac{n}{2}dS_{n-1}^{n-6}+S_{n}^{n-6}
\end{equation*}
\begin{equation*}
=\bigg[-\frac{n(n-1)(n-2)}{2^3}d^3S_{n-3}^{n-6}\bigg]+3\bigg[\frac{n(n-1)}{2^2}d^2S_{n-2}^{n-6}-\frac{n}{2}d S_{n-1}^{n-6} \bigg]+S_{n}^{n-6}
\end{equation*}

\begin{eqnarray*}
=-\frac{n(n-1)(n-2)}{2^3}d^3\bigg[-\frac{n-3}{2} dS_{n-4}^{n-7}+S_{n-3}^{n-7}\bigg]+3\frac{n(n-1)}{2^2}d^2\bigg(-\frac{n-2}{2} dS_{n-3}^{n-7}+S_{n-2}^{n-7}\bigg)\\[2mm]-3\frac{n}{2}d\bigg(-\frac{(n-1)}{2} dS_{n-2}^{n-7}+S_{n-1}^{n-7} \bigg)-\frac{n}{2}d S_{n-1}^{n-7}+S_{n}^{n-7}
\end{eqnarray*}

\begin{eqnarray*}
=\frac{n(n-1)(n-2)(n-3)}{2^4}d^4 S_{n-4}^{n-7}-4\frac{n(n-1)(n-2)}{2^3}d^3S_{n-3}^{n-7}+6\frac{n(n-1)}{2^2}d^2S_{n-2}^{n-7}-4\frac{n}{2}d S_{n-1}^{n-7}+S_{n}^{n-7}
\end{eqnarray*}

\begin{eqnarray*}
=-\frac{n(n-1)(n-2)(n-3)(n-4)}{2^5} d^5S_{n-5}^{n-8}+5\frac{n(n-1)(n-2)(n-3)}{2^4}d^4S_{n-4}^{n-8}\\[2mm]-10\frac{n(n-1)(n-2)}{2^3}d^3S_{n-3}^{n-8}+10\frac{n(n-1)}{2^2}d^2 S_{n-2}^{n-8}-5\frac{n}{2}d S_{n-1}^{n-8}+S_n^{n-8}
\end{eqnarray*}
\begin{equation*}
\cdots
\end{equation*}
\begin{equation*}
\cdots
\end{equation*}
\begin{equation*}
\cdots
\end{equation*}
\begin{equation*}
=\sum_{i=0}^{m}\binom m i \frac{\prod_{j=0}^{i-1}(n-j)}{2^i}(-1)^id^iS_{n-i}^{n-3-m}
=\sum_{i=0}^{m}\binom m i \bigg(\frac{d}{2}\bigg)^i\frac{n!}{(n-i)!}(-1)^iS_{n-i}^{n-3-m}
\end{equation*}
\end{proof}
\begin{theorem}Solutions of Theorem $2$ Or solutions of Equation $(3)$:
\begin{equation}
\sum_{r=1}^{k}(a+(r-1)d)^{n-1}=\frac{1}{nd}\bigg[\sum_{i=0}^{m}\binom m i \bigg(\frac{d}{2}\bigg)^i\frac{n!}{(n-i)!}(-1)^iS_{n-i}^{n-3-m}\bigg]
\end{equation}
\end{theorem}
\begin{proof}Easily followed from Cramer's rule of solving linear system of equations:
\begin{equation*}
\sum_{r=1}^{k}(a+(r-1)d)^{n-1}=\frac{(n-1)!d^{(n-1)}}{n!d^n}S^{(n-3)}_n=\frac{1}{nd}S^{(n-3)}_n
\end{equation*}
\begin{equation*}
=\frac{1}{nd}\bigg[\sum_{i=0}^{m}\binom m i \bigg(\frac{d}{2}\bigg)^i\frac{n!}{(n-i)!}(-1)^iS_{n-i}^{n-3-m}\bigg]
\end{equation*}
\begin{equation*}
\Longrightarrow \sum_{r=1}^{k}(a+(r-1)d)^{n-1}=\frac{1}{nd}\bigg[\sum_{i=0}^{m}\binom m i \bigg(\frac{d}{2}\bigg)^i\frac{n!}{(n-i)!}(-1)^iS_{n-i}^{n-3-m}\bigg]
\end{equation*}
\end{proof}
\begin{corollary}
\begin{equation}
\frac{d}{(n-1)!(n-3)!}\sum_{r=1}^{k}(a+(r-1)d)^{n-1}=\sum_{i=0}^{n-3}\frac{1}{i!(n-i)!(n-3-i)!} \bigg(\frac{d}{2}\bigg)^i(-1)^iS_{n-i}
\end{equation}
Where 
\begin{equation*}
S_{n-i}^0=S_{n-i}=\frac{n-i}{2}d^{n-i-1}J-\frac{n-i}{2}d^{n-i-2}J^2-d^{n-i-1}J+J^{n-i}
\end{equation*}
\begin{equation}
=\bigg(\frac{n-i}{2}-1\bigg)kd^{n-i}-\frac{n-i}{2}d^{n-i-2}((a+kd)^2-a^2)+(a+kd)^{n-i}-a^{n-i}
\end{equation}
\end{corollary}
\begin{proof}
Follows from equation $(5)$ for $m=n-3$.
\end{proof}
\begin{example} Solve for
\begin{equation*}
\sum_{r=1}^{k}r^{3000}
\end{equation*}
\end{example}
{\bf{Solution.}} From Equation $(6)$ for $a=d=1$, we have
\begin{equation*}
\sum_{r=1}^{k}r^{3000}=\frac{1}{3001}\bigg[\sum_{i=0}^{m}\binom m i \frac{1}{2^i}\frac{3001!}{(3001-i)!}(-1)^iS_{3001-i}^{3001-3-m} \bigg]
\end{equation*}
\begin{equation*}
=3000!\bigg[\sum_{i=0}^{2998}\binom {2998} i \frac{1}{2^i}\frac{1}{(3001-i)!}(-1)^iS_{3001-i}^{0} \bigg]
\end{equation*}
Where, for $i=0,1,\cdots 2998$, we have
\begin{equation*}
S_{3001-i}^0=S_{3001-i}=\bigg(\frac{3001-i}{2}-1\bigg)k-\frac{3001-i}{2}(k^2+2k)+(1+k)^{3001-i}-1
\end{equation*}
\begin{example}Let $d=1$ and $a=x+iy$, where $x,y\in\mathbb{R}$ and $i=\sqrt{-1}$
\begin{equation*}
\frac{1}{(n-1)!(n-3)!}\sum_{r=1}^{k}(x+iy+r-1)^{n-1}=\sum_{i=0}^{n-3}\frac{1}{i!(n-i)!(n-3-i)!} \bigg(\frac{1}{2}\bigg)^i(-1)^iS_{n-i}
\end{equation*}
Where 
\begin{equation*}
S_{n-i}=\bigg(\frac{n-i}{2}-1\bigg)k-\frac{n-i}{2}(2(x+iy)k+k^2)+(x+iy+k)^{n-i}-(x+iy)^{n-i}
\end{equation*}
\end{example}

\begin{theorem}Solutions of Theorem $3$.
\begin{equation}
\sum_{r=1}^{k}(-1)^{(r-1)}(a+(r-1)d)^{n-1}=\frac{1}{nd}\bigg[\sum_{i=0}^{m}\binom m i \bigg(\frac{d}{2}\bigg)^i\frac{n!}{(n-i)!}(-1)^{(i+1)}S_{n-i}^{n-3-m}\bigg]
\end{equation}
\end{theorem}
\begin{proof}
Please kindly see the proof of Theorem $2$.
\end{proof}
\begin{corollary}
\begin{equation}
\sum_{r=1}^{k}(-1)^{(r-1)}(a+(r-1)d)^{n-1}=\frac{1}{nd}\bigg[\sum_{i=0}^{n-3}\binom {n-3} i \bigg(\frac{d}{2}\bigg)^i\frac{n!}{(n-i)!}(-1)^{(i+1)}S_{n-i}\bigg]
\end{equation}
Where 
\begin{equation*}
S_{n-i}^0=S_{n-i}=\frac{n-i}{2}d^{n-i-1}J-\frac{n-i}{2}d^{n-i-2}J^2-d^{n-i-1}J+J^{n-i}
\end{equation*}
\begin{equation}
=\bigg(\frac{n-i}{2}-1\bigg)kd^{n-i}+\frac{n-i}{2}d^{n-i-2}((a+kd-d)^2-(a-d)^2)+(-1)^{(n-i-1)}[(a+kd-d)^{n-i}-(a-d)^{n-i}]
\end{equation}
\end{corollary}
\begin{proof}
Follows from equation $(8)$ for $m=n-3$.
\end{proof}



\newpage
\bibliographystyle{model1-num-names}

\begin{thebibliography}{09}
\bibitem{(549-551)}K. MacMillan and J. Sondow, Proofs of power sum and binomial coefficient congruences via Pascal’s identity, Amer. Math. Monthly $118 (2011) 549-551$
\bibitem{8}J. Riordan. Combinatorial Identities. New York: Wiley, $1968$.
\bibitem{2}L. Comtet. Advanced Combinatorics. Dordrecht: Reidel, $1974$.
\bibitem{1}P. Bachmann. Niedere Zahlentheorie $II$. New York: Chelsea, $1968$.
\bibitem{10}J. Wiener. "A Calculus Exercise for the Sums of Integer Powers." Math. Magazine $65(1992):249-51$.
\end{thebibliography}

\end{document}